\documentclass{article}
\usepackage{graphicx} % Required for inserting images

\pdfoutput=1 

\usepackage{authblk} % add additional author names and affiliations

\usepackage[top=2.54cm, bottom=2.54cm, left=3.17cm, right=3.17cm]{geometry}
\usepackage{amsmath}    % provides mathematics symbols
\usepackage{amsfonts}    % provides mathematics symbols
\usepackage{amsthm} % provides the environment proof
\usepackage{mathtools} % provides \coloneqq
\usepackage{amssymb} % provide \nsubseteq
\usepackage{enumitem} % enumerate
\usepackage{tikz} % tikz
\usepackage{subcaption} % provide subfigures
\usepackage{indentfirst} % provides indent
\usepackage{hyperref} % provide hyperref

\tikzset{every picture/.style={line width=0.75pt}} %set default line width to 0.75pt     

\newtheorem{theorem}{Theorem}[section] 
\newtheorem{lemma}[theorem]{Lemma}

\newtheorem{claim}[theorem]{Claim}
\newtheorem{definition}[theorem]{Definition}

\DeclareMathOperator*{\conv}{conv }
\DeclareMathOperator*{\coll}{coll}

\title{Helly-type theorems for separated $d$-intervals}
\author{Wei Rao\thanks{Moscow Institute of Physics and Technology, Institutskiy per. 9, 141700, Dolgoprudny, Russia. Email address: raowei1998@gmail.com}    
}
\date{}

\begin{document}

\maketitle

\begin{abstract}
    A separated $d$-interval is defined as a disjoint union of $d$ convex sets from the real line $\mathbb R$. In this paper, we establish a series of Helly-type theorems for convexity spaces derived from separated \(d\)-intervals. Our results encompass the Radon number, Helly number, colorful Helly number, fractional Helly number, colorful fractional Helly theorem, $(p,q)$ theorem, and two kinds of colorful $(p,q)$ theorems for these convexity spaces. The primary tools employed in our proofs involve simplicial complexes and collapsibility.
\end{abstract}


\section{Introduction}
\subsection{Helly-type theorems and $d$-collapsibility}
In 1923, Helly~\cite{helly1923mengen} published a foundational theorem on the intersection patterns of convex sets in $\mathbb R^d$, where $d$ is a positive integer. Specifically, for a finite family \(\mathcal{F}\) of convex sets in \(\mathbb{R}^d\), Helly's theorem states that if every at most \(d+1\) members of \(\mathcal{F}\) have a non-empty intersection, then all members of \(\mathcal{F}\) intersect. 

Over the years, this result has been generalized and extended in various directions. Notable examples include the colorful Helly theorem, first proven by Lovász and later fully detailed by B{\'a}r{\'a}ny~\cite{barany1982generalization}; the fractional Helly theorem introduced by Katchalski and Liu~\cite{katchalski1979problem}; and the colorful fractional Helly theorem, also developed by B{\'a}r{\'a}ny et al.~\cite{barany2014colourful}. Other significant advancements include the \((p,q)\) theorem by Alon and Kleitman~\cite{alon1992piercing}, as well as two variations of the colorful \((p,q)\) theorems by B{\'a}r{\'a}ny and Matoušek \cite{barany2003fractional} and B{\'a}r{\'a}ny et al. \cite{barany2014colourful}, respectively. For more details, refer to \cite{barany2022helly, barany2021combinatorial}. Rather than immediately delving into these theorems, we first introduce the concepts of abstract simplicial complex and $d$-collapsibility, which was introduced by Wegner in \cite{wegner1975d}. The $d$-collapsibility is a powerful tool for establishing a series of Helly-type theorems. 

Let $V$ be a finite set. A collection $K$ of subsets of $V$ is an \textit{abstract simplicial complex} if for any $\alpha \in K$ and $\beta \subseteq \alpha$, we have $\beta \in K$. Elements of $K$ are called \textit{faces} of $K$. The \textit{dimension of a face} $\alpha \in K$ is defined as $\dim \alpha \coloneqq |\alpha|-1$. The \textit{dimension of a complex} $K$ is defined as $\dim K \coloneqq\max \{ \dim \alpha : \alpha \in K \}$. Let $I$ be a finite set of indices and $\mathcal F$ be a family of sets $(C_i)_{i \in I}$. The \textit{nerve} of $\mathcal F$ is $K(\mathcal F) \coloneqq \{J \subseteq I: \bigcap_{j\in J} C_j \neq \emptyset, J \text{ is finite} \}$. One can easily verify that the nerve is an abstract simplicial complex.

A face $\sigma \in K$ is \textit{free} if it is contained in a unique inclusion-maximal face.  An \textit{elementary $d$-collapse} of $K$ is a step in which a free face $\sigma$ with $\dim \sigma \leq d-1$ and all faces containing $\sigma$ are removed. Moreover, we denote the result by $\coll(K, \sigma) =  K \setminus \{\tau \in K: \sigma \subseteq \tau \}$. A simplicial complex $K$ is \textit{$d$-collapsible} if $K$ can be reduced to the empty complex by a sequence of elementary $d$-collapses. For more details on the notion related to $d$-collapsibility and the consequence of $d$-collapsibilty, refer to \cite{tancer2013intersection}. In this paper, we shall use the following recent result.

\begin{theorem}[The optimal colorful fractional Helly theorem for $d$-collapsible complexes \cite{bulavka2021optimal}]\label{theorem: optimal colorful fractional Helly for simplicial complexes}
    Let $K$ be a $d$-collapsible simplicial complex with the set of vertices $N=N_1 \sqcup \dots \sqcup N_{d+1}$ divided into $d+1$ disjoint subsets of sizes $|N_i|=n_i$, respectively. If $K$ contains at least $\alpha n_1 \dots n_{d+1}$ colorful $d$-faces, where $\alpha \in (0,1]$, that is, faces $\sigma$ with $|\sigma \cap N_i|=1$ for every $i \in [d+1]$. Then there is an index $i\in [d+1]$ such that the dimension of the induced subcomplex on $N_i$ is at least $(1-(1-\alpha)^{1/(d+1)})n_i-1$, that is, $\dim K[N_i] \geq (1-(1-\alpha)^{1/(d+1)})n_i-1$.
\end{theorem}

\subsection{Separated $d$-intervals}
Now, we are ready to introduce our research subjects and some related results. A \textit{homogeneous $d$-interval} is a union of at most $d$ convex sets from the real line $\mathbb R$. A \textit{separated/nonhomogeneous $d$-interval} is a homogeneous $d$-interval consisting of $d$ convex set components $I^{(1)},\dots,I^{(d)}$ such that $I^{(i+1)} \subseteq (i, i+1)$ for $0\leq i\leq d-1$, where $(i, i+1)$ is an open interval. For convenience, we employ the following equivalent definition.

\begin{definition}\label{de8}
    A separated $d$-interval $I$ is a disjoint union of $d$ convex sets from the real line $\mathbb R$, that is, $I = \bigsqcup_{i \in [d]} I^{(i)}= \bigcup_{i \in [d]} \{(x, i) \in \mathbb R \times [d]: x\in I^{(i)} \subseteq \mathbb R \}$, where all $I^{(i)}$ are convex sets in $\mathbb R$. Moreover, we say $I^{(i)}$ is the $i$-th level of $I$ and any point $(x,i)$, where $x \in \mathbb R$, is in the $i$-th level.
\end{definition}

The reader may realize that there is some relation between the axis-parallel boxes in $\mathbb R^d$ and the separated $d$-intervals. For more details, refer to Section~\ref{section: relation with axis-parallel boxes}.

Let $\mathcal F$ be a finite family of subsets on the ground set $X$. The \textit{transversal number} of $\mathcal F$, denoted by $\tau(\mathcal F)$, is the minimum integer $k$ such that there is a set $S \subseteq X$ of size at most $k$ such that $S \cap C \neq \emptyset$ for any $C \in \mathcal F$. The \textit{matching number} of $\mathcal F$, denoted by $\nu(\mathcal F)$, is the maximum integer $k$ such that there is a subfamily $\mathcal F' \subseteq \mathcal F$ of size $k$, consisting of pairwise disjoint sets. Tardos \cite{tardos1995transversals} and Kaiser \cite{kaiser1997transversals} proved the following results.

\begin{theorem}[Tardos \cite{tardos1995transversals}, Kaiser \cite{kaiser1997transversals}]\label{tar}
    If $\mathcal F$ is a finite family of homogeneous $d$-intervals, then $\tau(\mathcal F) \leq (d^2-d+1) \nu (\mathcal F)$. If $\mathcal F$ is a finite family of separated $d$-intervals, then $\tau(\mathcal F) \leq (d^2-d) \nu (\mathcal F)$.
\end{theorem}

Then Frick and Zerbib \cite{frick2019colorful} established the colorful version. Recently, a sparse colorful version and a matroid colorings version were proven in \cite{mcginnis2024sparse} and \cite{mcginnis2024matroid}, respectively.

\subsection{Abstract convexity spaces and main results}
In this paper, we shall study the convexity spaces constructed on separated $d$-intervals. Moreover, throughout the paper, we will use the term $d$-intervals to refer to separated $d$-intervals.

\begin{definition}
    A convexity space is a set system $\mathcal C$ on a ground set $X$ that satisfies the following three properties:
    \begin{enumerate}
        \item $\emptyset, X \in \mathcal C$;
        \item If $\emptyset \neq \mathcal D \subseteq \mathcal C$, then $\bigcap \mathcal D \in \mathcal C $;
        \item If $\emptyset \neq \mathcal D \subseteq \mathcal C$ is totally ordered by inclusion, then $\bigcup \mathcal D \in \mathcal C$.
    \end{enumerate}
\end{definition}

Let $(X, \mathcal C)$ be a convexity space. For any subset $Y \subseteq X$, the \textit{convex hull} of $Y$ is defined as $\conv Y= \bigcap \{K: K \in \mathcal C, Y \subseteq K \}$. We call the members of $\mathcal C $ \textit{convex sets}. For example, when $X = \mathbb R^d$ and $\mathcal C (\mathbb R^d)$ is the set of all the usual convex sets in $\mathbb R^d$, space $(\mathbb R^d, \mathcal C(\mathbb R^d))$ is the standard convexity space, which is widely studied in convex geometry and combinatorial convexity. For a comprehensive introduction to the convexity space, refer to \cite{van1993theory, holmsen2024helly}.

Let $\mathbb R \times [d]=\bigcup_{i \in [d]} \{(x, i): x \in \mathbb R \}$ and $P \subseteq \mathbb R \times [d]$. Let $\mathcal C_{\equiv}(P)$ be the set of all subsets of $P$ of the form $I \cap P$ where $I$ is a $d$-interval, that is, \[ \mathcal C_{\equiv}(P) = \{I \cap P: I \subseteq \mathbb R \times [d] \text{ is a $d$-interval}\} . \]

One can easily verify that $(P,\mathcal C_{\equiv} (P))$ is a convexity space. Note that when $P = \mathbb R \times [d]$, set $\mathcal C_{\equiv} (P)$ is the set of usual separated $d$-intervals defined above, which has been extensively studied; see \cite{bjorner2017using}. Throughout the paper, we denote the convexity spaces constructed in this manner with respect to some $P\subseteq \mathbb R \times [d]$ by $(P,\mathcal C_{\equiv}(P))$ and call them \textit{$\equiv$-convexity space}. 

To the best of our knowledge, the collapsibility of $d$-intervals has not been studied in prior works. One of our results is about the collapsibility of $d$-intervals and can be stated as follows.

\begin{theorem}\label{theorem: collapsibility}
    Let $P \subseteq \mathbb R \times [d]$ and $(P,\mathcal C_{\equiv} (P))$ be a $\equiv$-convexity space. For any finite family $\mathcal C \subseteq \mathcal C_{\equiv} (P)$, the nerve of $\mathcal C$ is $(2d-1)$-collapsible.
\end{theorem}

This result establishes a series of Helly-type theorems, which shall be stated latter. Combining the consequences of Theorem~\ref{theorem: collapsibility} and some improvements gives another main result -- Theorem~\ref{theorem: Helly-type theorems} in Section~\ref{section: Helly-type notions}, which contains main Helly-type theorems that we want to show.

This paper is organized as follows. In Section~\ref{section: Helly-type notions}, we shall introduce all the aforementioned Helly-type theorems in the context of convexity spaces. Then we state all Helly-type theorems for convexity spaces constructed on separated $d$-intervals. In Section $3$, we state all the necessary lemmas that we need to use or prove. Moreover, we introduce the proof chain and the results that are needed to be proven. From Section $4$ to Section $9$, we give the proofs. In Section $10$, we discuss some results that we do not want to put in the main body. Moreover, we discuss the relation between separated $d$-intervals and axis-parallel boxes.

\section{Helly related notions and main results}\label{section: Helly-type notions}

Once the concept of convexity spaces is established, many Helly-type notions can be generalized to these spaces. One fundamental concept is the Radon partition and its associated Radon number, introduced in Radon's theorem \cite{radon1921mengen}, which is commonly employed in the proof of Helly's theorem.

\begin{definition}[Radon partition and Radon number]
    Let $(X, \mathcal C)$ be a convexity space and $Y \subseteq X$. A Radon partition of $Y$ is a partition $Y = A \cup B$ such that $\conv A \cap  \conv B \neq \emptyset$. The Radon number $r(X,\mathcal C)$ is the minimal integer $n$ (if it exists) such that every subset $Y \subseteq X$ of size at least $n$ has a Radon partition.
\end{definition}

In the standard convexity space $(\mathbb R^d, \mathcal C(\mathbb R^d))$, Radon's theorem \cite{radon1921mengen} asserts that $r(\mathbb R^d, \mathcal C(\mathbb R^d)) = d+2$. Onn \cite{onn1991geometry} and Sierksma \cite{sierksma1976relationships} proved that $5 \cdot 2^{d-2} +1 \leq r(\mathbb Z^d, \mathcal C(\mathbb Z^d) ) \leq d(2^d-1)+3$, where $\mathcal C(\mathbb Z^d)$ is the set of all subsets of $\mathbb Z^d$ of the form $C \cap \mathbb Z^d$ where $C \subseteq \mathbb R^d$ is a convex set, that is, \[ \mathcal C(\mathbb Z^d) = \{C \cap \mathbb Z^d: C\subseteq \mathbb R^d \text{ is a convex set}\} . \] Edwards and Sober{\'o}n \cite{edwards2024extensions} proved $r(P, \mathcal C_{\Box}(P)) \leq 2d+1$, where $P \subseteq \mathbb R^d$ and $\mathcal C_{\Box}(P)$ is the set of all subsets of $P$ of the form $B \cap P$ where $B \subseteq \mathbb R^d$ is an axis-parallel box, that is, \[ \mathcal C_{\Box}(P) = \{B \cap P: B\subseteq \mathbb R^d \text{ is an axis-parallel box}\} . \]
By \textit{axis-parallel box} (or simply a \textit{box}), we mean a set $B$ in $\mathbb R^d$ that is the Cartesian product of $d$ non-empty convex sets of $\mathbb R$. Moreover, when $P = \mathbb R^d$, Eckhoff \cite{eckhoff2000partition} showed that $r(\mathbb R^d, \mathcal C_{\Box}(\mathbb R^d)) = \Theta (\log d)$.

\begin{definition}[Helly number]\label{de1}
    Let $(X, \mathcal C)$ be a convexity space. The Helly number $h(X, \mathcal C)$ is the minimal integer $n$ (if it exists) such that for any finite family $\mathcal F \subseteq \mathcal C$ of convex sets, if every at most $n$ members of $\mathcal F$ intersect, then all members intersect. 
\end{definition}

Helly's theorem asserts $h(\mathbb R^d, \mathcal C(\mathbb R^d)) = d+1$. Doignon \cite{doignon1973convexity} proved that $h(\mathbb Z^d, \mathcal C(\mathbb Z^d)) = 2^d$. It is known that $h(\mathbb R^d, \mathcal C_{\Box}(\mathbb R^d)) =2$. Halman \cite{halman2008discrete} demonstrated that $h(P, \mathcal C_{\Box}(P)) \leq 2d$.

\begin{definition}[Colorful Helly number]\label{de2}
    Let $(X, \mathcal C)$ be a convexity space. Finite subfamilies $\mathcal F_1,\dots, \mathcal F_n$ of $\mathcal C$ have the colorful Helly property if every colorful $n$-tuples of them intersect, that is, $\bigcap_{i=1}^n C_i \neq \emptyset$ for all $C_1 \in \mathcal F_1 , \dots, C_n \in \mathcal F_n$. 
    
    The colorful Helly number $h_c(X, \mathcal C)$ is the minimal integer $n$ (if it exists) such that for any $n$ finite families $\mathcal F_1,\dots ,\mathcal F_n \subseteq \mathcal C$ of convex sets, if they have the colorful Helly property, then there exists a family whose intersection is non-empty. 
\end{definition}

The colorful Helly's theorem asserts that $h_c(\mathbb R^d, \mathcal C(\mathbb R^d)) = d+1$. It is known that $h_c(\mathbb R^d, \mathcal C_{\Box}(\mathbb R^d)) = d+1$. Edwards and Sober{\'o}n \cite{edwards2024extensions} proved that $h_c(P, \mathcal C_{\Box}(P)) \leq 2d$.

\begin{definition}[Fractional Helly number]\label{de3}
    Let $(X, \mathcal C)$ be a convexity space. We say $(X, \mathcal C)$ admits a fractional Helly theorem for $k$-tuples, if there exist an integer $k$ and a function $\beta : (0,1)\to (0,1)$ such that every finite family $\mathcal F \subseteq \mathcal C$ with at least $\alpha \binom{|\mathcal F|}{k}$ intersecting $k$-tuples, contains an intersecting subfamily of size at least $\beta(\alpha)|\mathcal F|$. 
    
    The fractional Helly number $h_f(X, \mathcal C)$ is the minimal integer $n$ (if it exists) such that $(X,\mathcal C)$ admits the fractional Helly theorem for $n$-tuples. 
\end{definition}

The fractional Helly theorem asserts $h_f(\mathbb R^d, \mathcal C(\mathbb R^d)) = d+1$. B{\'a}r{\'a}ny and Matou{\v{s}}ek \cite{barany2003fractional} proved that $h_f(\mathbb Z^d, \mathcal C(\mathbb Z^d)) = d+1$. It is known that $h_f(\mathbb R^d, \mathcal C_{\Box}(\mathbb R^d)) = d+1$. Edwards and Sober{\'o}n \cite{edwards2024extensions} proved $h_f(P, \mathcal C_{\Box}(P)) = d+1$. Combining the colorful Helly theorem and the fractional Helly theorem leads to the colorful fractional Helly theorem, which takes the following form.

\begin{definition}[Colorful fractional Helly theorem]\label{de4}
    Let $(X, \mathcal C)$ be a convexity space. We say $(X, \mathcal C)$ admits a colorful fractional Helly theorem for colorful $k$-tuples, if there exist an integer $k$ and a function $\beta : (0,1)\to (0,1)$ such that for every finite families $\mathcal F_1,\dots, \mathcal F_k \subseteq \mathcal C$ of sizes $n_1,\dots,n_k$ with at least $\alpha n_1\dots n_k$ intersecting colorful $k$-tuples, there is a family $\mathcal F_i$ that contains an intersecting subfamily of size at least $\beta(\alpha)|\mathcal F_i|$. 
\end{definition}

Bulavka, Goodarzi, Tancer \cite{bulavka2021optimal} established an optimal colorful fractional Helly theorem for $(\mathbb R^d,\mathcal C(\mathbb R^d))$. Moreover, any convexity space $(X,\mathcal C)$, satisfying the fractional Helly property for $k$-tuples, admits the following $(p,q)$ theorem for all $p \geq q \geq k$ \cite{alon2002transversal}.

\begin{definition}[$(p,q)$ theorem]\label{de5}
    Let $(X,\mathcal C)$ be a convexity space and $p\geq q$ be integers. A finite family $\mathcal F \subseteq \mathcal C$ has the $(p,q)$ property if among any $p$ members of $\mathcal F$ some $q$ of them intersect. We say $(X,\mathcal C)$ admits a $(p,q)$ theorem if there exists an integer $N=N(p,q)$ such that for every finite family $\mathcal F \subseteq \mathcal C$ with the $(p,q)$ property, there is a set $S \subseteq X$ of size at most $N$ such that any member $C \in \mathcal F$ intersects $S$.
\end{definition}

\begin{theorem}[Alon et al. \cite{alon2002transversal}]\label{fh2pq}
    Let $(X,\mathcal C)$ be a convexity space and $p\geq q$ be integers. If $(X,\mathcal C)$ admits a fractional Helly theorem for $k$-tuples, it also admits a $(p,q)$ theorem for all $p\geq q\geq k$.
\end{theorem}

Alon and Kleitman \cite{alon1992piercing} demonstrated that $(\mathbb R^d, \mathcal C(\mathbb R^d))$ admits a $(p,q)$ theorem for $p\geq q \geq d+1$. B{\'a}r{\'a}ny and Matou{\v{s}}ek \cite{barany2003fractional} proved that $(\mathbb Z^d, \mathcal C(\mathbb Z^d))$ admits a $(p,q)$ theorem for $p\geq q \geq d+1$. Edwards and Sober{\'o}n \cite{edwards2024extensions} extended this result to $(P, \mathcal C_{\Box}(P))$, proving that it also admits a $(p,q)$ theorem for $p\geq q \geq d+1$. 

Now we present two different kinds of colorful $(p,q)$ theorems by combining the colorful Helly theorem and the $(p,q)$ theorem. For a positive integer $n$, we define $[n] \coloneqq \{1,\dots, n \}$.

\begin{definition}[The first kind of colorful $(p,q)$ theorem]\label{de6}
    Let $(X,\mathcal C)$ be a convexity space and $p,q$ be positive integers such that $p\geq q$. Finite families $\mathcal F_1, \dots,\mathcal F_q \subseteq \mathcal C$ have the first kind of colorful $(p,q)$ property if whenever we choose, for each $i\in [q]$, distinct sets $C_{i, 1},\dots, C_{i, p} \in \mathcal F_i$, there are subscripts $j_1,\dots, j_{q} \in [p]$ such that $\bigcap_{i=1}^{q} C_{i, j_i}  \neq \emptyset$. 
    
    We say $(X,\mathcal C)$ admits the first kind of colorful $(p,q)$ theorem if there exists an integer $N_c=N_c(p,q)$ such that for every finite families $\mathcal F_1,\dots,\mathcal F_q \subseteq \mathcal C$ with the first kind of colorful $(p,q)$ property, there is a set $S \subseteq X$ of size at most $N_c$ and an index $i \in [q]$ such that any member $C_i \in \mathcal F_i$ intersects $S$.
\end{definition}

B{\'a}r{\'a}ny and Matou{\v{s}}ek \cite{barany2003fractional} proved that $(\mathbb R^d, \mathcal C(\mathbb R^d))$ admits the first kind of colorful $(p,q)$ theorem with $p\geq q\geq d+1$. Moreover, there is a second kind of colorful $(p,q)$ theorem, which is formulated as follows.

\begin{definition}[The second kind of colorful $(p,q)$ theorem]\label{de7}
    Let $(X,\mathcal C)$ be a convexity space and $p,q$ be positive integers such that $p\geq q$. Finite families $\mathcal F_1, \dots,\mathcal F_p \subseteq \mathcal C$ have the second kind of colorful $(p,q)$ property if whenever we choose, for each $i\in [p]$, $C_i \in \mathcal F_i$, there are $q$ of them contain a common point. 
    
    We say $(X,\mathcal C)$ admits the second kind of colorful $(p,q)$ theorem if there exists an integer $M_c=M_c(p,q)$ such that for every finite families $\mathcal F_1,\dots,\mathcal F_p \subseteq \mathcal C$ with the second kind of colorful $(p,q)$ property, there is a set $S \subseteq X$ of size at most $M_c$ and an index $i \in [p]$ such that any member $C_i \in \mathcal F_i$ intersects $S$.
\end{definition}

B{\'a}r{\'a}ny et al. \cite{barany2014colourful} proved that $(\mathbb R^d, \mathcal C(\mathbb R^d))$ admits the second kind of colorful $(p,q)$ theorem with $p\geq q\geq d+1$. 

Based on above definitions, all Helly-type theorems for separated $d$-intervals that we shall prove can be stated as follows.

\begin{theorem}\label{theorem: Helly-type theorems}
    Let $P \subseteq \mathbb R \times [d]$ and $(P,\mathcal C_{\equiv} (P))$ be a $\equiv$-convexity space.
    \begin{enumerate}
        \item\label{item1} The Radon number of $(P,\mathcal C_{\equiv} (P))$ satisfies $r(P,\mathcal C_{\equiv} (P)) \leq 2d+1$. Moreover, if $P$ contains at least $2$ points in every level, then $r(P,\mathcal C_{\equiv} (P)) = 2d+1$.
        \item\label{item2} The Helly number of $(P,\mathcal C_{\equiv} (P))$ satisfies $h(P,\mathcal C_{\equiv} (P)) \leq  2d$. Moreover, if $P$ contains at least $2$ points in every level, then $h(P,\mathcal C_{\equiv} (P)) = 2d$.
        \item\label{item3} The colorful Helly number of $(P,\mathcal C_{\equiv} (P))$ satisfies $h_c(P,\mathcal C_{\equiv} (P)) \leq 2d$. Moreover, if $P$ contains at least $2$ points in every level, then $h_c(P,\mathcal C_{\equiv} (P)) = 2d$.
        \item\label{item4} The fractional Helly number of $(P,\mathcal C_{\equiv} (P))$ satisfies $ h_f(P,\mathcal C_{\equiv} (P)) = 2$.
        \item Convexity space $(P,\mathcal C_{\equiv} (P))$ admits a $(p,q)$ theorem with $p\geq q \geq 2$.
        \item Convexity space $(P,\mathcal C_{\equiv} (P))$ admits the first kind of colorful $(p,q)$ theorem with $p\geq q \geq 2$.
        \item Convexity space $(P,\mathcal C_{\equiv} (P))$ admits the second kind of colorful $(p,q)$ theorem with $p\geq q \geq 2$.
    \end{enumerate}
\end{theorem}

\begin{theorem}[Colorful fractional Helly theorem for $\equiv$-convexity space]\label{cfh}
    Let $P \subseteq \mathbb R \times [d]$ and $(P,\mathcal C_{\equiv} (P))$ be a $\equiv$-convexity space. Let $\alpha \in (0,1]$ and $\beta = 1- (1-\alpha)^{1/2d}$. Let $\mathcal C_1, \dots, \mathcal C_{2d}$ be finite subfamilies of $\mathcal C_{\equiv} (P)$. If there are at least $\alpha |\mathcal C_1|\dots |\mathcal C_{2d}|$ colorful $2d$-tuples that intersect, then some $\mathcal C_i$ contains a subfamily of size $\beta |\mathcal C_i|$, whose elements intersect.
\end{theorem}

Moreover, we prove generalizations with respect to the $k$-intersecting. For an arbitrary number of sets in $\mathcal C_{\equiv}(P)$, we say they \textit{$k$-intersect}, where $k \in [d]$, if their intersection contains at least $k$ points of $P$ that are from $k$ distinct levels. We then generalize the colorful Helly theorem and the fractional Helly theorem of $(P,\mathcal C_{\equiv} (P))$ with respect to $k$-intersecting. We replace the usual intersection with the $k$-intersection in Definition~\ref{de2} and Definition~\ref{de3}, then we have the following results.

\begin{theorem}\label{theorem: colorful helly for k-intersecting}
    Let $P \subseteq \mathbb R \times [d]$, $(P,\mathcal C_{\equiv} (P))$ be a $\equiv$-convexity space and $k \in [d]$ be a positive integer. The colorful Helly number of $(P,\mathcal C_{\equiv} (P))$ with respect to $k$-intersecting, denoted by $h_{ck}(P,\mathcal C_{\equiv} (P))$, satisfies $h_{ck}(P, \mathcal C_{\equiv} (P)) \leq 2d-k+1 $.
\end{theorem}

\begin{theorem}\label{theorem: fractional helly for k-intersecting}
    Let $P \subseteq \mathbb R \times [d]$, $(P,\mathcal C_{\equiv} (P))$ be a $\equiv$-convexity space and $k \in [d]$ be a positive integer. The fractional Helly number of $(P,\mathcal C_{\equiv} (P))$ with respect to $k$-intersecting, denoted by $h_{fk}(P,\mathcal C_{\equiv} (P))$, satisfies $h_{fk}(P,\mathcal C_{\equiv} (P)) \leq 2d-k+1$.
\end{theorem} 

\section{Preliminaries}
By Theorem~\ref{theorem: optimal colorful fractional Helly for simplicial complexes}, Theorem~\ref{theorem: collapsibility} implies Theorem~\ref{cfh}. Subsequently, Theorem~\ref{cfh} provides the upper bounds of the Helly number, the colorful Helly number and the fractional Helly number. By Theorem~\ref{fh2pq}, the fourth item of Theorem~\ref{theorem: Helly-type theorems} implies the fifth item of Theorem~\ref{theorem: Helly-type theorems}. For the sixth and seventh items, we need to introduce the following variation of the colorful Helly theorem, which was first proven for $(\mathbb Z^d, \mathcal C(\mathbb Z^d))$ by B{\'a}r{\'a}ny and Matou{\v{s}}ek in \cite{barany2003fractional}.

\begin{definition}[Partial colorful Helly number]\label{definition: partial colorful Helly number}
    Let $(X, \mathcal C)$ be a convexity space. We say $(X,\mathcal C)$ admits a partial colorful Helly theorem for colorful $k$-tuples, if there exist an integer $k$ and a number $N_{kpc}=N_{kpc} (m)$ such that for every finite families $\mathcal F_1,\dots, \mathcal F_k \subseteq \mathcal C$ of sizes $|\mathcal F_i| = N_{kpc}(m)$ with the colorful Helly property, there is a family $\mathcal F_i$ that contains an intersecting subfamily of size at least $m$.
    
    The partial colorful Helly number $h_{pc}(X, \mathcal C)$ is the minimal integer $n$ (if it exists) such that $(X,\mathcal C)$ admits the partial colorful Helly theorem for $n$-tuples.
\end{definition}

Then we shall prove the following four lemmas, which implies that the fourth, sixth, and seventh items of Theorem~\ref{theorem: Helly-type theorems} hold.

\begin{lemma}\label{lemma: partial colorful Helly number}
    Let $P \subseteq \mathbb R \times [d]$ and $(P,\mathcal C_{\equiv} (P))$ be a $\equiv$-convexity space. The partial colorful Helly number of $(P,\mathcal C_{\equiv} (P))$ is $2$.
\end{lemma}

\begin{lemma}\label{1cpq}
    Let $(X,\mathcal C)$ be a convexity space and $p,q$ be positive integers such that $p\geq q$. If $ \max \{ h_{pc}(X,\mathcal C),  h_f(X,\mathcal C) \}=k$, then $(X,\mathcal C)$ admits the first kind of colorful $(p,q)$ theorem with $p\geq q \geq k$.
\end{lemma}

\begin{lemma}\label{lemma: colorful fractional Helly number}
    Let $P \subseteq \mathbb R \times [d]$ and $(P,\mathcal C_{\equiv} (P))$ be a $\equiv$-convexity space. Convexity space $(P,\mathcal C_{\equiv} (P))$ admits the colorful fractional Helly theorem for colorful $2$-tuples.
\end{lemma}

\begin{lemma}\label{2cpq}
    Let $(X,\mathcal C)$ be a convexity space and $k,p,q$ be positive integers such that $p\geq q$. If $(X,\mathcal C)$ admits a colorful fractional Helly theorem for colorful $k$-tuples, then $(X,\mathcal C)$ admits the second kind of colorful $(p,q)$ theorem with $p\geq q \geq k$.
\end{lemma}

Moreover, we introduce the fractional transversal number, fractional matching number, and some related theorems, which will be used to prove Lemma~\ref{1cpq} and Lemma~\ref{2cpq}.

\begin{definition}
    Let $\mathcal F$ be a finite set family on the ground set $X$. The fractional transversal number $\tau^{\ast}(\mathcal F)$ is the minimum of $\sum_{x\in X} f(x)$ over all functions $f: X \to [0,1]$ such that $\sum_{x\in C} f(x) \geq 1$ for any $C \in \mathcal F$.
\end{definition}

\begin{definition}
    Let $\mathcal F$ be a finite set family on the ground set $X$. The fractional matching number $\nu^{\ast}(\mathcal F)$ is the maximum of $\sum_{C\in \mathcal F} f(C)$ over all functions $f: \mathcal F \to [0,1]$ such that $\sum_{C \in \mathcal F: x \in C} f(C) \leq 1$ for any $x \in X$.
\end{definition}

It is known that $\nu^{\ast} (\mathcal F) = \tau ^{\ast}(\mathcal F)$. Moreover, Alon et al. \cite{alon2002transversal} proved the following theorem.

\begin{theorem}[Alon, Kalai, Matou{\v{s}}ek, Meshulam \cite{alon2002transversal}]\label{bou1}
    Let $(X,\mathcal C)$ be a convexity space. If $(X,\mathcal C)$ admits the fractional Helly theorem, then for every finite family $\mathcal F \subseteq \mathcal C$, we have $\tau(\mathcal F) \leq f(\tau^{\ast}(\mathcal F))$ for some function $f$.
\end{theorem}

Then the following result can be used to bound $\tau^{\ast} (\mathcal F)$.
\begin{theorem}[Alon, Kleitman \cite{alon1992piercing}]\label{bou2}
    Let $\mathcal F$ be a finite family of sets and $\gamma > 0$. Then $\nu^{\ast}(\mathcal F) \leq \gamma$ if and only if every blown-up copy of $\mathcal F$, say $\mathcal F^{\ast}$, which contains $m(C) \in \mathbb N$ copy of each $C \in \mathcal F$,  contains an intersecting subfamily of size at least $\gamma^{-1}|\mathcal F^{\ast}|$.
\end{theorem}

In summary, to establish all our results, it suffices to prove Theorem~\ref{theorem: collapsibility}, the Radon number in Theorem~\ref{theorem: Helly-type theorems}, the lower bounds of the Helly number and colorful Helly number in special cases, Lemma~\ref{lemma: partial colorful Helly number}, Lemma~\ref{1cpq}, Lemma~\ref{lemma: colorful fractional Helly number}, Lemma~\ref{2cpq}, Theorem~\ref{theorem: colorful helly for k-intersecting}, and Theorem~\ref{theorem: fractional helly for k-intersecting}. We will begin by proving Theorem~\ref{theorem: colorful helly for k-intersecting} and Theorem~\ref{theorem: fractional helly for k-intersecting}, as their key ideas are instrumental in the proof of Theorem~\ref{theorem: collapsibility}.

\section{Proofs of Theorem~\ref{theorem: colorful helly for k-intersecting} and Theorem~\ref{theorem: fractional helly for k-intersecting}}
The general approach follows the well-known strategy used to prove the colorful Helly theorem and the fractional Helly theorem for the standard convexity space $(\mathbb R^d, \mathcal C(\mathbb R^d))$; refer to the second proof of Theorem 12.1 for the colorful Helly theorem and the proof of Theorem 9.1 for the fractional Helly theorem in \cite{barany2021combinatorial}. 

The core is to find some function such that the maximum of every non-empty intersection that we concern is attained by a unique ``point'' and the intersection of some bounded number of sets among all sets forming the intersection also has this maximum. Our function is constructed in the following way.

Let $\hat{\mathcal C}_{\equiv} (P)$ be the family of sets consisting of all intersections of finite subfamilies of $\mathcal C_{\equiv} (P)$ that are $k$-intersecting, that is, \[ \hat{\mathcal C}_{\equiv} (P) \coloneqq \{\bigcap \mathcal C : \mathcal C \subseteq \mathcal C_{\equiv} (P), \mathcal C \text{ is $k$-intersecting} \} . \] 
Define a function $f: \hat{\mathcal C}_{\equiv} (P) \to \overline{\mathbb R}^d$, where $\overline{\mathbb R} = \mathbb R \cup  \{ - \infty , + \infty\}$, such that $f(\hat{C})= (x_1,\dots,x_d)$, where  \[  x_i = \begin{cases}   
    \sup \hat{C}^{(i)} & \text{ if } \hat{C}^{(i)} \neq \emptyset, \\
    -\infty & \text{ if } \hat{C}^{(i)} = \emptyset,
\end{cases} \]
and $\hat{C}^{(i)}$ is the $i$-th level of $d$-interval $\hat{C}$ as defined in Definition~\ref{de8}.

\begin{lemma}\label{lem2}
    Let $n$ be a positive integer. If $n$ sets of $\mathcal C_{\equiv} (P)$, say $ \mathcal D = \{ C_1, \dots, C_n \}$, $k$-intersect, then there is a subfamily of size at most $2d-k$, say $\mathcal D'$, such that $f(\bigcap \mathcal D) = f(\bigcap \mathcal D')$. 
\end{lemma}

\begin{proof}[Proof of Lemma~\ref{lem2}]
    Since $\mathcal D$ is finite, we may assume that the sets in $\mathcal D$ are compact. Hence, we have $x_i \neq + \infty$ for all $i\in [d]$ and $\sup \hat{C}^{(i)} = \max \hat{C}^{(i)} \in \hat{C}^{(i)}$. We construct $\mathcal D'$ based on $f(\bigcap \mathcal D)= (a_1,\dots,a_d)$ in the following way. 
    
    For $i\in [d]$, if $a_i =- \infty$, then we have $\bigcap_{j=1}^n C_{j}^{(i)} =\emptyset$. Note that it implies $\bigcap_{j=1}^n I(C_{j}^{(i)}) \cap P =\emptyset$, where $I(C) \coloneqq \{\bigcap I : C \subseteq I, I \subseteq \mathbb R \times [d] \text{ is a $d$-interval} \}$ throughout the paper. Since all sets are compact, by the definition of $I(C)$, two endpoints of $I(C_{j}^{(i)})$ are in $C_{j}^{(i)} \subseteq P$. Then we have $\bigcap_{j=1}^n I(C_{j}^{(i)}) = \emptyset$, otherwise two endpoints of $\bigcap_{j=1}^n I(C_{j}^{(i)})$ are in $P$, which contradicts $\bigcap_{j=1}^n I(C_{j}^{(i)}) \cap P =\emptyset$. By Helly's theorem in $\mathbb R$, there are two of them, say $I(C_{j_1}^{(i)}), I(C_{j_2}^{(i)})$, such that $I(C_{j_1}^{(i)}) \cap  I(C_{j_2}^{(i)}) = \emptyset$. Then we include $C_{j_1}$ and $C_{j_2}$ in $\mathcal D'$.
    
    If $a_i \neq -\infty$, then there exists $C_{j}^{(i)}$ such that $\sup C_{j}^{(i)} = a_i$. We include $C_{j}$ in $\mathcal D'$. Note that $f(\bigcap  \mathcal D ) =f(\bigcap \mathcal D')$. Moreover, since $\mathcal D$ $k$-intersects, there are at most $2(d-k) + k = 2d-k$ sets that are included in $\mathcal D'$.
\end{proof}

\subsection{Proof of Theorem~\ref{theorem: colorful helly for k-intersecting}}
Let $\mathcal C_1,\dots,\mathcal C_{2d-k+1}$ be finite subfamilies of $\mathcal C_{\equiv} (P)$, which satisfy the colorful Helly property with respect to $k$-intersecting. We shall show that there exists an index $i\in[2d-k+1]$ such that $\mathcal C_i$ is $k$-intersecting.

Among all colorful $2d-k$ tuples, denoted by $\hat{\mathcal C}$, consider $ \{f(\hat{C}) :  \hat{C} \in \hat{\mathcal C}  \}$ and choose the lexicographically minimal one, say $f(\hat{C_0})$. Without loss of generality, we may assume that $\hat{C_0} = \bigcap_{i=1}^{2d-k} C_i$, where $C_i \in \mathcal C_i$. Consider the first $k$ coordinates of $f(\hat{C_0})$ that are not $- \infty$, say $a_{i_j}$ in $i_j$-th level, where $j\in [k]$ and $i_1< \dots <i_k$. We claim that all sets of $\mathcal C_{2d-k+1}$ contain $k$ points $(a_{i_j} , i_j)$, which leads to $h_{ck}(P,\mathcal C_{\equiv} (P)) \leq 2d-k+1$.

Indeed, assume for the sake of contradiction that there is $C_{2d-k+1} \in \mathcal C_{2d-k+1}$ that avoids one of these $k$ points. Then note that $f(\bigcap_{i=1}^{2d-k+1} C_i) <_{\text{lex}} f(\hat{C_0})$, where $ <_{\text{lex}} $ ($ \leq_{\text{lex}} $, respectively) means less than (less than or equal to, respectively) with respect to the lexicographical order throughout the paper. By Lemma~\ref{lem2}, there exist at most $2d-k$ sets among $C_1,\dots,C_{2d-k+1}$, without loss of generality, say $C_2,\dots,C_{2d-k+1}$, such that $f(\bigcap_{i=2}^{2d-k+1} C_i) = f(\bigcap_{i=1}^{2d-k+1} C_i) <_{\text{lex}} f(\hat{C_0})$, contradicting the minimality of $f(\hat{C_0})$.

\subsection{Proof of Theorem~\ref{theorem: fractional helly for k-intersecting}}
Let $\alpha \in (0,1]$ and $\mathcal C$ be a finite subfamily of $\mathcal C_{\equiv} (P)$ of size $n$. We claim that if there are at least $\alpha \binom{n}{2d-k+1}$ $k$-intersecting $2d-k+1$ tuples, then there exists a $k$-intersecting subfamily $\mathcal C' \subseteq \mathcal C$ of size at least $\frac{\alpha}{2d-k+1} n$, which implies that $h_{fk}(P,\mathcal C_{\equiv} (P)) \leq 2d-k+1$.

Since $\mathcal C$ is finite, we may assume that sets are compact. By Lemma~\ref{lem2}, there are at least $\frac{\alpha \binom{n}{2d-k+1}}{\binom{n}{2d-k}}$ distinct $2d-k+1$ tuples, denoted by $ \tilde{\mathcal C}$, such that for any $\mathcal C' \in \tilde{\mathcal C}$ we have $f(\bigcap \mathcal C') = f(\bigcap \mathcal C_0)$, where $\mathcal C_0$ is some $k$-intersecting $2d-k$ tuple. It implies that there are at least
$$
\begin{aligned}
    \frac{\alpha \binom{n}{2d-k+1}}{\binom{n}{2d-k}} + 2d-k  &= \alpha \frac{n-2d+k}{2d-k+1} +2d-k   \\
    &  \geq \frac{\alpha}{2d-k+1} n 
\end{aligned}
$$
sets that contain all $(a_{i_j},i_j) $, where $a_{i_j}$ are the non-infinite coordinates of $f(\bigcap \mathcal C_0)$. Since $\mathcal C_0$ is $k$-intersecting and sets are compact, there are at least $k$ non-infinite coordinates, which implies that those sets are $k$-intersecting.

\section{Proof of Theorem~\ref{theorem: collapsibility}}
The general approach follows the proof in \cite{wegner1975d}. Instead of sweeping along a line, we sweep along the disjoint lines in lexicographic order. Let $f=(f_0, f_1,f_2,\dots)$ be the $f$-vector of $K(\mathcal C)$, that is, $f_k$ is the number of $k$-dimensional simplices of $K(\mathcal C)$. We prove the theorem by induction on $m \coloneqq \sum_{i=0}^{2d-1}f_i$. When $m=0$, we have $K(\mathcal C) = \emptyset$, which is $(2d-1)$-collapsible. Then the inductive argument is given by the following lemma.

\begin{lemma}\label{lem1}
    If non-empty complex $K$ is the nerve of some family $\mathcal C \subseteq \mathcal C_{\equiv} (P)$, then there is a free face $\sigma$ of $K$ with $\dim \sigma \leq 2d-2$ such that $\coll (K,\sigma)$ is also the nerve of some family $\mathcal G \subseteq \mathcal C_{\equiv} (P)$.
\end{lemma}

\begin{proof}[Proof of Lemma~\ref{lem1}]
    Let $\mathcal C$ be a finite subfamily of $\mathcal C_{\equiv} (P)$ such that $K$ is the nerve of $\mathcal C$. Since $\mathcal C$ is finite, we may assume that sets are compact. Let $\hat{\mathcal C}$ be the family of sets consisting of all non-empty intersections of members of $\mathcal C$, that is, $\hat{\mathcal C} \coloneqq \{\bigcap \mathcal C' : \mathcal C' \subseteq \mathcal C, \bigcap \mathcal C' \neq \emptyset \}$. Then the members of $\hat{\mathcal C}$ correspond to the simplices of $K$. With a slight abuse of notation, sometimes we may consider $\hat{C} \in \hat{\mathcal C}$ as the set of members of $\mathcal C$ used to construct the intersection instead of the intersection itself.

    Define function $f: \hat{\mathcal C} \to \overline{\mathbb R}^d$, where $\overline{\mathbb R} = \mathbb R \cup  \{ - \infty , + \infty\}$, such that $f(\hat{C})= (x_1,\dots,x_d)$, where  \[  x_i = \begin{cases}   
    \max \hat{C}^{(i)} & \text{ if } \hat{C}^{(i)} \neq \emptyset, \\
    -\infty & \text{ if } \hat{C}^{(i)} = \emptyset.
\end{cases} \]

    Consider $\hat{C}_0= C_1\cap \dots\cap C_n \in \hat{\mathcal C}$, where $C_i \in \mathcal C$ for $i\in [n]$, such that $f(\hat{C}_0)$ is lexicographically minimal among the image of $f$ and $n$ is minimal. Suppose that the corresponding simplex of $\hat{C}_0$ is $\sigma$. We claim that $\sigma$ is the desired free face. 
    
    First, note that by Lemma~\ref{lem2}, we have $n \leq 2d-1$, which implies that $\dim \sigma \leq 2d-2$. Then we shall show that any $C \in \mathcal C$ with $C \cap \hat{C}_0 \neq \emptyset$ contains $(a_i,i)$, where $a_i$ is the first coordinate of $f(\hat{C}_0)$ such that $a_i \neq - \infty$. It implies that $C_1,\dots,C_n$ are contained in a unique inclusion-maximal $\mathcal C' \subseteq \mathcal C$ such that $\bigcap \mathcal C' \neq \emptyset$. Hence, the corresponding face $\sigma$ is free. 
    
    Assume for the sake of contradiction that there exists $C \in \mathcal C$ with $C \cap \hat{C}_0 \neq \emptyset$ such that $C$ does not contain $(a_i ,i)$. Let $f(C \cap \hat{C}_0) =(b_1,\dots,b_d)$. Note that $b_1=\dots=b_{i-1} = -\infty$ and $b_i < a_i$. It implies that $f(C \cap \hat{C}_0)$ is lexicographically less than $f( \hat{C}_0)$, which contradicts the minimality of $f(\hat{C}_0)$.

    It remains to show that $\coll(K,\sigma)$ is the nerve of some family $\mathcal G \subseteq \mathcal C_{\equiv} (P)$. When $n=1$, note that the nerve of $\mathcal C \setminus C_1$ is $\coll(K,\sigma)$. When $n\geq 2$, note that the family obtained from $\mathcal C$ by replacing $C \in \mathcal C$ by \[ C \setminus \Big( \{(x,i):x \leq a_i \} \cup  \bigcup_{j=i+1}^d \{(x,j): x\in \mathbb R \} \Big) \] for $k\in [n]$ has $\coll(K, \sigma)$ as its nerve, where $a_i$ is the first coordinate of $f(\hat{C}_0)$ such that $a_i \neq - \infty$.
\end{proof}

\section{Proof of Lemma~\ref{1cpq}}
The proof follows the same approach as the proof of Theorem 4.1 in \cite{barany2003fractional}. What we do is only point out that the proof works for the abstract settings and summarize the exact conditions. We adopt the format from proof of Theorem 29.1 in \cite{barany2021combinatorial}. In order to prove Lemma~\ref{1cpq}, we need the following Erd{\H{o}}s-Simonovits theorem \cite{erdHos1983supersaturated}.

\begin{theorem}\label{erd}
    For any positive integers $k, t$ and any $c>0$ there exists $c' >0$ such that a $k$-uniform hypergraph on $n$ vertices with at least $c n^k$ edges contains at least $c' n^{kt}$ copies of $K^k(t)$, that is, complete $k$-uniform $k$-partite hypergraph with each class of size $t$.
\end{theorem}
Observe that when $c$ is fixed and $n$ is large enough, there is at least one copy of $K^ k(t)$.

Note that it is enough to show that it holds for $q=k$. Since $h_f(X,\mathcal C) = k$, by Theorem~\ref{bou1}, it is enough to show that $\tau^{\ast} < \infty$ for some $i$ provided that $\mathcal F_1,\dots, \mathcal F_k$ have the first kind of colorful $(p,k)$ property. By linear programming duality, we have $\tau ^{*}(\mathcal F_i) = \nu^{*}(\mathcal F_i)$. Hence it is enough to show that $\nu ^{*} (\mathcal F_i)$ is bounded by a constant that depends only on $p$ and $k$ for some $i \in [k]$. The fractional matching number $\nu^{*} (\mathcal F_i)$ is the value of the following linear programming

\begin{align*}
   & \nu^{*} (\mathcal F_i) = \max \sum_{C_{ij} \in \mathcal F_i} w_i( C_{i, j}) , \\
   & \text{subject to } \sum_{C_{i, j} \in \mathcal F_{i}, p \in C_{i, j}} w_i (C_{i, j}) \leq 1 (\forall x \in X) \text{ and } w_i : \mathcal F_i \to [0,1].
\end{align*}

Let $w _i : \mathcal F_i \to [0,1]$ for which $\nu ^{*} (\mathcal F_i)$ is attained. Since $\mathcal F_i$ are finite, we may assume that the values of $w_i$ are all rational. Let $w_i (C_{i, j}) = \frac{n_{C_{i, j}}}{m_i}$, where $m_i$ is a large enough common denominator of the rationals $w_i(C_{i, j})$. For each family $\mathcal F_i$, let $\mathcal F_{i}^{*}$ be family consisting of $n_{C_{i, j}}$ copies of $C_{i, j}$ for each $C_{i, j} \in \mathcal F_i$. Set $n_i = \sum_{C_{i, j} \in \mathcal F_i} n_{C_{i, j}}$. Now we shall show that there is $i\in [k]$ and $x \in X$ such that $x$ is common to at least $\beta n_i$ elements of $\mathcal F_{i}^*$ for some $\beta > 0$, since then we have \[ 1\geq \sum_{C_{i, j} \in \mathcal F_{i}, x \in C_{i, j}} w_i (C_{i, j}) \geq \beta \frac{n_i}{m_i} = \beta \nu^*(\mathcal F_i),\] which implies $\nu^*(\mathcal F_i) \leq \frac{1}{\beta}$.

Since $h_f(X,\mathcal F) = k$, it is enough to show that the conditions of the fractional Helly theorem are satisfied for some $\mathcal F_i^*$. We form a $k$-uniform $k$-partite hypergraph $\mathcal H$. Its vertices in the $i$-th class are the sets in $\mathcal F_i^*$, and its edges are $k$ tuples of sets from distinct classes that contain a common point.

Let $s$ be an large enough integer that is divisible by $k$. Consider a set of size $s$ from each class, that is, $\mathcal D_i \subseteq \mathcal F_i^*$ of size $s$ for all $i \in [k]$.

\begin{claim}
    For any such $\mathcal D_1,\dots, \mathcal D_{k}$ some $\mathcal D_i$ contains an $k$-tuple of sets that contain a common point.
\end{claim}
\begin{proof}
    If some $\mathcal D_i$ contains $k$ copies of the same set, we are done. Hence we may assume that every family contains at most $k-1$ copies of the same set. Then there is $\mathcal D_i' \subseteq \mathcal D_i$ of distinct sets of size $s^* = \frac{s}{k}$. Since for any choice of subfamilies $\mathcal D''_i \subseteq \mathcal D_i'$ of size $p$, there is an edge of $\mathcal H$. It leads to that there are at least $\binom{s^*}{p}^{k}$ such edges and each edge is counted $\binom{s^* -1}{p-1}^{k}$ times. Then there are at least $(s^*/p)^{k}$ edges at the subhypergraph $\mathcal H'$ induced by $\mathcal D'_i$. 

    By Theorem~\ref{erd}, $\mathcal H'$ contains a complete $k$-partite hypergraph with each class of size $N_{kpc}(k)$, where $N_{kpc}(k)$ is the needed size of families in the partial colorful Helly theorem for $k$-tuples and existence of intersecting subfamily of size $k$, since $s$ is large enough. Then by the partial colorful Helly theorem, there is at least one class that contains a $k$-tuple whose intersection contains a common point.
\end{proof}

Then the total number of such $k$-tuples is at least \[ \binom{n_1}{s} \binom{n_2}{s}  \dots \binom{n_{k}}{s} \] and each tuple in the $i$-th class is counted \[ \binom{n_1}{s} \dots \binom{n_i - k }{s-k}  \dots \binom{n_{k}}{s} \] times. Assume $\mathcal F_i^*$ contains $M_i$ such tuples. Then we have \[ \prod_{i=1}^{k} \binom{n_i}{s} = \sum_{i=1}^{k} M_i \binom{n_1}{s} \dots \binom{n_i - k }{s-k}  \dots \binom{n_{k}}{s}  ,\] showing that \[ 1= \sum_{i=1}^{k}M_i \frac{\binom{n_i - k }{s-k}}{\binom{n_i  }{s}} = \sum_{i=1}^{k}M_i \frac{\binom{s }{k}}{\binom{n_i  }{k}} .\] It implies that some $M_i \geq \alpha \binom{n_i}{k}$ with $\alpha = \frac{1}{k} \binom{s}{k}^{-1} >0$, completing the proof.

\section{Proof of Lemma~\ref{2cpq}}
The proof is the same as the proof of Theorem $3$ in \cite{barany2014colourful}. Note that it is enough to show it for $q=k$. Moreover, by Theorem~\ref{bou1} and Theorem~\ref{bou2}, it suffices to show that for every blown-up copy $\mathcal F_1^*,\dots, \mathcal F_{p}^*$, there is some $\mathcal F_i^*$ containing a subfamily of size $\gamma^{-1} |\mathcal F_i^*|$ for some $\gamma >0$, whose elements contain a common point. Suppose $\beta(\alpha)$ is the function in colorful fraction Helly theorem for $(X,\mathcal C)$. Set $\delta = \binom{p}{k}^{-1}$ and $\gamma = (\beta(\delta))^{-1}$. 

Let $\mathcal H$ be a complete $p$-uniform $p$-partite hypergraph, whose vertices in the $i$-th class are the sets in $\mathcal F_i^*$ for $i\in [p]$. For any edge $e= (C_1, \dots, C_p) \in \mathcal H$ and $J \subseteq [p]$, denote the partial edge $(C_j : j \in J)$ by $e(J)$. For $I \in \binom{[p]}{k}$, let $\mathcal H(I)$ be a $k$-uniform $k$-partite hypergraph whose classes are $\mathcal F_i^*$, where $i\in [I]$, and $f=(C_i : i \in I)$ is an edge of $\mathcal H(I)$ if $\bigcap_{i \in I} C_i  \neq \emptyset$. Let $N_i = |\mathcal F_i|$.

\begin{claim}\label{counting}
    Some $\mathcal H(I)$ contains at least $\delta \prod_{i\in I} |N_i|$ edges.
\end{claim}

\begin{proof}
    Let $N= N_1 \dots N_p$. Let $(e,f)$ be the pair such that $e \in \mathcal H$ and $f=e(I) \in \mathcal H(I)$. Since for every colorful $p$-tuples there are $k$ of them containing a common point, for every $e \in \mathcal H$ there is an $I \in \binom{[p]}{k}$ such that $e(I) \in \mathcal H(I)$. It implies that
    $$
    \begin{aligned}
        N & \leq \text{ number of such pairs } (e,f) \\
        & = \sum_{ I \in \binom{[p]}{k}} \sum_{f \in \mathcal H(I)} | \{e \in \mathcal H : f = e(I) \} | \\
        & = \sum_{ I \in \binom{[p]}{k}} \sum_{f \in \mathcal H(I)} \prod_{j \notin I} N_j \\
        & = N \sum_{ I \in \binom{[p]}{k}}  \frac{|\mathcal H(I)|}{ \prod_{i \in I} N_i}.
    \end{aligned}
    $$
    Hence we have $\frac{|\mathcal H(I)|}{ \prod_{i \in I} N_i} \geq \delta$ for some $I$, completing the proof.
\end{proof}

Then by colorful fractional Helly theorem for $(X,\mathcal C)$, some $\mathcal F_i^*$ contains an intersecting subfamily of size $\gamma^{-1} |\mathcal F_i^*|$, completing the proof.

\section{Proof of Lemma~\ref{lemma: partial colorful Helly number} and Lemma~\ref{lemma: colorful fractional Helly number}}

Note that it is enough to prove Lemma~\ref{lemma: colorful fractional Helly number}, since Lemma~\ref{lemma: colorful fractional Helly number} implies Lemma~\ref{lemma: partial colorful Helly number} by setting $N_{2pc}(m) = \frac{m}{\beta(1)}$, where $\beta(\alpha)$ is the corresponding function in the colorful fractional Helly theorem for $2$-tuples.

Let $\mathcal C_1, \mathcal C_2 \subseteq  \mathcal C_{\equiv} (P)$ be finite families of size $n_1,n_2$, respectively. Suppose that there are $\alpha n_1n_2$ intersecting colorful $2$-tuples. Note that there exists a level at which at least $\frac{\alpha}{d} n_1 n_2$ intersecting $2$-tuples intersect. Then by Theorem~\ref{cfh} for the case of $d=1$, one family contains a intersecting subfamily of size $\beta n$, where $\beta = 1-(1-\frac{\alpha}{d})^{1/2}$, completing the proof.

\section{Proofs of first three items of Theorem~\ref{theorem: Helly-type theorems}}
\subsection{Proof of \ref{item1} of Theorem~\ref{theorem: Helly-type theorems}}
We first prove that $r(P,\mathcal C_{\equiv} (P)) \leq 2d+1$, that is, for any subset $A \subseteq P$ with $|A| \geq 2d+1$, there is a partition $A = X \cup Y$ such that $\conv X \cap \conv Y \neq \emptyset$.

Since $|A| \geq 2d+1$, by the pigeonhole principle there are at least $3$ points in the same level, say $(x_1,i),(x_2,i),(x_3,i)$. Without loss of generality, we may assume that $x_1 \leq x_2 \leq x_3$. Then for arbitrary partition of $A'= A \setminus\{ (x_1,i),(x_2,i),(x_3,i)\}= X' \cup Y'$, let $X= X' \cup \{(x_2,i) \}$ and $Y= Y' \cup \{(x_1,i), (x_3,i) \}$. Note that $X,Y$ is the desired partition.

Then we prove that $r(P,\mathcal C_{\equiv} (P)) > 2d$, if $P$ contains at least $2$ points in every level. Consider $A= \bigcup_{i\in [d]} \{a_{i1}, a_{i2} \}$, where $a_{i1},a_{i2}$ are two points in the $i$-th level. Note that $A$ does not have a Radon partition. 

\subsection{Proof of \ref{item2} and \ref{item3} of Theorem~\ref{theorem: Helly-type theorems}}
For the Helly number and the colorful Helly number, it is enough to prove the lower bound for the special cases as stated in the Preliminaries.

For the Helly number, consider family $\mathcal F = \{C_1,\dots,C_{2d} \}$, where \[  C_k = \begin{cases}   
    \Big(\bigcup_{j\in [d], j\neq \lceil k/2 \rceil} \conv \{a_{j1},a_{j2} \} \Big) \cup \{a_{\lceil k/2 \rceil1} \}, & \text{ if $k$ is odd} , \\
     \Big( \bigcup_{j\in [d], j\neq \lceil k/2 \rceil} \conv \{a_{j1},a_{j2} \} \Big) \cup \{a_{\lceil k/2 \rceil2} \}, & \text{ if $k$ is even},
\end{cases} \]
and $a_{i1},a_{i2}$ are points in the $i$-th level, that is, $C_k$ contain some one point in the $\lceil k/2 \rceil$-th level and convex hulls of two points in the other levels. Note that every $2d-1$ of them intersect, but all of them do not intersect.

For the colorful Helly number, it is known that $h_c(X,\mathcal C) \geq h(X,\mathcal C)$. Hence, we have $h_c(X,\mathcal C) \geq 2d$, if $P$ contains at least $2$ points in every level.

\section{Discussion}\label{section: conclusion}
\subsection{Missed parts}
Except for the Helly-type results listed in Section~\ref{section: Helly-type notions}, one can verify that an analogous result of Theorem $4$ in \cite{barany2014colourful} also arises directly from the colorful fractional Helly theorem by the same combinatorial argument in the proof of Theorem $4$ in \cite{barany2014colourful}. 

Additionally, we also can prove an analogous result of Theorem $1.3$ in \cite{frankl2025helly} for separated $d$-intervals by taking the following $2d$ orders $<_i$ for $i\in [2d]$. For any two separated $d$-intervals $I_1,I_2$, when $i$ is odd, we say $I_1 <_i I_2$ if the left-endpoint of the $\lceil \frac{i}{2}\rceil$-th level of $I_1$ is to the right of the left-endpoint of the $\lceil \frac{i}{2}\rceil$-th level of $I_2$. When $i$ is even, we say $I_1 <_i I_2$ if the right-endpoint of the $\frac{i}{2}$-th level of $I_1$ is to the left of the right-endpoint of the $ \frac{i}{2}$-th level of $I_2$. Then we can prove some Helly-type theorems for the so-called monotone properties introduced in that paper.

While Theorem~\ref{tar} and its colorful version yield improved bounds for the $(p,q)$ theorem and the second kind of colorful $(p,q)$ theorem for the space of usual separated $d$-intervals, that is $(\mathbb R \times [d],\mathcal C_{\equiv} (\mathbb R \times [d]))$, obtaining the first kind of colorful $(p,q)$ theorem remains unknown within the framework of existing results for $d$-intervals.

However, the following fact is implicit in the proof of Lemma~\ref{lem2}: the nerve complex of a finite family $\mathcal C \subseteq \mathcal C_{\equiv} (P)$ can be represented as the nerve complex of a family of separated $d$-intervals by replacing each $C \in \mathcal C$ by $I(C)$, which is defined in the proof of Lemma~\ref{lem2}. Then it implies that Theorem~\ref{tar} and its variations also hold for $\mathcal C \subseteq \mathcal C_{\equiv} (P)$, although the original proof is topological and can not be applied for discrete version directly.

\subsection{Relation with axis-parallel boxes}\label{section: relation with axis-parallel boxes}
We may always regard the separated $d$-interval as a box in $\mathbb R^d$. The only thing that we need to be careful is that when we deal with an empty level, we need to regard the corresponding box as being ``far from'' the others. Then the $k$-intersecting of separated $d$-intervals looks like axis-parallel $(d-k)$-transversal, that is, a collection of axis-parallel $(d-k)$-flats that intersects all members.

If we only concern the case of $(\mathbb R \times [d], \mathcal C_{\equiv}(\mathbb R \times [d]))$, there is some result about the axis-parallel flats transversal for boxes due to Chakraborty, Ghosh and Nandi \cite{chakraborty2025stabbing}. If we concern the case of $(P,\mathcal C_{\equiv} (P))$, there is some difference between them. Specifically, we can not claim that Theorem~\ref{theorem: colorful helly for k-intersecting} implies the following analogous statement: given a set $L$ of axis-parallel $(d-k)$-flat and a finite family of axis-parallel boxes in $\mathbb R^d$. If every $2d-k+1$ members can be pierced by an axis-parallel $(d-k)$-flat from $L$, then all boxes can be pierced by an axis-parallel $(d-k)$-flat from $L$. 

The problem is that we can arbitrarily arrange the points in $P$ to form some $k$-intersecting, but we can not combine some $(d-k)$-flat to form a new $(d-k)$-flat. Then maybe it will be interesting to study these ``discrete'' problems for axis-parallel boxes.

\subsection{Fractional Helly number and dual VC-dimension}
In \cite{holmsen2024fractional}, Holmsen and Pat{\'a}kov{\'a} proved the following result, which is for the separable convexity space\footnote{Please refer to \cite{holmsen2024fractional} for the definitions.}.

\begin{theorem}[Holmsen and Pat{\'a}kov{\'a} \cite{holmsen2024fractional}]
    Let $(X, \mathcal C)$ be a separable convexity space with bounded Radon number. Suppose that the system of halfspaces has dual VC-dimension $d$. Then the fractional Helly number for $\mathcal C$ is at most $d+1$.
\end{theorem}

Since $(P,\mathcal C_{\equiv} (P))$, where $P \subseteq \mathbb R \times [d]$, is separable and has bounded Radon number, we attempted to apply it for getting fractional Helly number. However, the dual VC-dimension of the system of halfpsaces of $(P,\mathcal C_{\equiv} (P))$ is at least $\log_2 d$. It is far from the actual fractional Helly number $2$; refer to \ref{item4} of Theorem~\ref{theorem: Helly-type theorems}. Hence, this can serve as an example to show that the fractional Helly number is not always equal to the sum of dual VC-dimension of the system of halspaces and one.

\section*{Acknowledgments}
I thank Alexander Polyanskii for inspiring this research problem and Rahul Gangopadhyay for his comments on the manuscript.

\bibliographystyle{alpha}
\bibliography{references}

\end{document}